\documentclass[12pt]{amsart}
\usepackage{amsfonts,amssymb,amscd,amsmath,enumerate,url,verbatim}
\usepackage[all]{xy}

\theoremstyle{plain}

\newtheorem{theorem}{Theorem}[section]

\newtheorem{prop}[theorem]{Proposition}

\newtheorem{corollar}[theorem]{Corollary}

\newtheorem{obs}[theorem]{Observation}
\theoremstyle{definition}

\newtheorem{construction}[theorem]{Construction}
\newtheorem{definition}[theorem]{Definition}
\newtheorem{example}[theorem]{Example}

\usepackage{tkz-graph}
\usetikzlibrary{arrows}

\newcommand{\m}{\mathfrak{m}}

\begin{document}
\title[On the existence of non-free totally reflexive modules]
{On the existence of non-free totally reflexive modules}

\author{J. Cameron Atkins and Adela Vraciu}
\address{Cameron Atkins\\Department of Mathematics\\University of South Carolina\\\linebreak
Columbia\\SC 29208\\ U.S.A}\email{atkinsj6@email.sc.edu}
\address{Adela Vraciu\\Department of Mathematics\\University of South Carolina\\\linebreak 
Columbia\\ SC 29208\\ U.S.A.} \email{vraciu@math.sc.edu}


\thanks{Research partly supported by NSF grant  DMS-1200085}
\maketitle
\begin{abstract}
For a standard graded Cohen-Macaulay ring $R$, if the quotient $R/(\underline{x})$ admits non-free totally reflexive modules, where $\underline{x}$ is a system of parameters consisting of elements of degree one, then so does the ring $R$. A non-constructive proof of this statement was given in \cite{Ta}. We give an explicit construction of the totally reflexive modules over $R$ obtained from those over $R/(\underline{x})$.  

We consider the question of which Stanley-Reisner rings of graphs admit non-free totally reflexive modules and discuss some examples.

For an Artinian local ring $(R, \m)$ with $\m^3 =0$ and containing the complex numbers, we describe an explicit construction of uncountably many non-isomorphic indecomposable totally reflexive modules, under the assumption that at least one such non-free module exists.
\end{abstract}

\section{Introduction}

Totally reflexive modules were introduced by Auslander and Bridger in \cite{AB}, under the name of modules of Gorenstein dimensin zero. These modules were used as a generalization of free modules, in order to define a new homological dimension for finitely generated modules over Noetherian rings, called the G-dimension. Over a Gorenstein ring, the totally reflexive modules are exactly the maximal Cohen-Macaulay modules, and Gorenstein rings are characterized by the fact that every finitely generated module has finite G-dimension.

In \cite{CPST} it was shown that one can use totally reflexive modules to give a characterization of simple hypersurface singularities among all complete local algebras. It was also shown in \cite{CPST} that if a local ring is not Gorenstein, then it either has infinitely many indecomposable pairwise non-isomorphic totally reflexive modules, or else it has none other than the free modules. This dichotomy points out that it is important to understand which non-Gorenstein rings admit non-free totally reflexive modules and which do not. This question is not well understood at present. In this paper we study this issue from the poit of view of reducing to the case of Artinian rings, and we use this technique to study a class of rings obtained from a graph.

In this paper, $R$ and $S$ will denote commutative Noetherian rings.

\begin{definition}
A finitely generated module $M$ is called {\em totally reflexive} if there exists an infinite complex of finitely generated free $R$-modules
$$
F: \ \ \ \ \cdots \rightarrow F_1 \rightarrow F_0 \rightarrow F_{-1} \rightarrow \cdots 
$$
such that $M$ is isomorphic to $\mathrm{Coker}(F_1 \rightarrow F_0)$, and such that both the complex $F$ and the dual $F^*=\mathrm{Hom}_R(F, R)$ are exact. 

Such a complex $F$ is called a {\em totally acyclic complex}. We say that $F$ is a {\em minimal totally acyclic complex} if the entries of the matrices representing the differentails are in the maximal ideal (or homogeneous maximal ideal in the case of a graded ring).
\end{definition}

It is obvious that free modules are totally reflexive. The next easiest example is provided by exact zero divisors, studied under this name in \cite{HS}:
\begin{definition}
A pair of elements $a, b \in R$ is called a {\em pair of exact zero divisors} if $\mathrm{Ann}_R(a)=(b)$ and $\mathrm{Ann}_R(b)=a$.

Note that if $R$ is an Artinian ring, then one of these conditions implies the other (as it implies that $l((a)) +l((b))=l(R)$).
\end{definition}
If $a, b$ is a pair of exact zero divisors, then the complex
$$
\cdots R \stackrel{a}{\rightarrow} R \stackrel{b}{\rightarrow} R \stackrel{a}{\rightarrow} \cdots
$$
is a totally acyclic complex, and $R/(a)$, $R/(b)$ are totally reflexive modules.

More complex totally reflexive modules can be constructed using a pair of exact zero divisor, see \cite{Ho} and \cite{CJRSW}, \cite{CGT}. 

Many properties of commutative Noetherian rings can be reduced to the case of Artinian rings, via specialization. We use this approach in order to study the existence of non-free totally reflexive modules. The following  is well-known (see Proposition 1.5 in \cite{CFH}).
\begin{obs}\label{first_obs}
Let $R$ be a Cohen-Macaulay ring and let $M$ be a non-free totally reflexive $R$-module. If $\underline{x}$ is a system of parameters in $R$, then $M/(\underline{x})M$ is a non-free totally reflexive $R/(\underline{x})$-module.
\end{obs}

On the other hand, Avramov, Gasharov and Peeva in \cite{AGP} give a construction of a module of complete intersection dimension zero (which implies totally reflexive)  over  any ring which is an embedded deformation. In \cite{AY}, infinitely many distinct isomorphism classes of such modules over a ring which is an embedded deformation are constructed. Thus we have:
\begin{theorem}\label{second_obs}\cite{AGP}, \cite{AY}
Let $(R, \m)$ be a Cohen-Macaulay module, and let $\underline{x} \subseteq \m^2$ be a $R$- regular sequence. The the quotient $R/(\underline{x})$ admits non-free totally reflexive modules.
 \end{theorem}

Given a Cohen-Macaulay standard graded or local  ring $(R, \m)$, one would like to investigate whether $R$ admits non-free totally reflexive modules via investigating the same issue for specializations $R/(\underline{x})$. In order to use this approach, one  needs a converse of Observation~\ref{first_obs}. In  light of Theorem ~\ref{second_obs},  such a converse cannot be true if the system of parameters $\underline{x}$ is contained in $\m ^2$ (as in this case $R/(\underline{x})$ always has non-free totally reflexive modules, even if $R$ does not). This converse is proved in~(\cite{Ta},  Proposition4.6):
\begin{theorem}\cite{Ta}\label{main_theorem}
Let $(S, \mathfrak{m})$ be a local ring, and let $x_1, \ldots, x_d$ be a regular sequence such that $x_i \in \mathfrak{m} \, \backslash \, \mathfrak{m}^2$. Let $R=S/(x_1, \ldots, x_d)$. If $R$ has non-free totally reflexive modules, then so does $S$.
\end{theorem}
The proof in~(\cite{Ta}) is non-constructive; in  Section 2 we give a constructive approach to this result in the graded case, where we indicate how a minimal totally acyclic complex over $R$ can be used to build a minimal totally acyclic complex over $S$.

Once we have reduced to an Artinian ring, the easiest to detect totally reflexive modules are provided by pairs of exact zero-divisors. Exact zero-divisors do not usually exist in the original ring of positive dimension. We will give examples where they can be found in specializations, thus allowing us to conclude that the original ring also had non-free totally reflexive modules. The examples that we focus on in Section 3 are Stanley-Reisner rings of connected  graphs. These are two-dimensional Cohen-Macaulay rings, and, after modding out by a linear system of parameters, they satisfy $\m ^3=0$ . We will give some necessary conditions for the existence of non-free totally reflexive modules, as well as examples where we can find pairs of exact zero divisors in the specialization. 

Most of the constructions of totally reflexive modules in literature start with a  pair of exact zero divisors, which can then be used to construct more complicated modules. We are only aware of one example (Proposition 9.1 in \cite{CJRSW}) of a ring which admits non-free totally reflexive modules, but does not have exact zero divisors. This example occurs over a characteristic two field, and can be considered a pathological case (the ring defined by the same equations over a field of characteristic different from two will have exact zero divisors). In Section 4 we provide another, characteristic-free example of a ring that does not have exact zero divisors, but has non-free totally reflexive modules. Moreover, we indicate how to construct infinitely many non-indecomposable non-isomorphic totally reflexive modules over this ring. It is likely that our example can be generalized to a family of rings with these properties.

In Section 5 we consider Artinian local rings $(R, \m)$ with $\m^3=0$ which contains the complex numbers, and we describe a construction that gives rise to uncountably many non-isomorpic indecomposable totally reflexive modules, under the assumption that one such non-free module exists.
It has been known from \cite{CPST} that, under the assumptions above, there would be infinitely many such modules, but, to the best of our knowledge, this is the first time that an explicit construction is provided that does not use a pair of exact zero-divisors. 

{\em Acknowledgemets:} We thank Ryo Takahashi for pointing us to his paper (\cite{Ta}) and Hailong Dao for introducing us to the results of (\cite{DT}).

\section{Lifting totally acyclic complexes}

In this section we let $S=k \oplus S_1 \oplus \cdots $ be a standard graded ring, and $x_1, \ldots, x_d \in S_1$ be a regular sequence. We let $R=S/(x_1, \ldots, x_d)$. It is known from (\cite{Ta}, Proposition 4.6) that if $R$ has non-free totally reflexive modules, then so does $S$. In this section we provide an explicit construction for totally acyclic complexes over $S$ that give rise to such modules, using totally acyclic complexes over $R$ as a starting point.
\begin{construction}\label{construction}
Given 
\begin{equation}\label{complex1} \cdots \longrightarrow R^{b_{i+1}} \stackrel{\delta_{i+1}}{\longrightarrow}R^{b_i}\stackrel{\delta_i}{\longrightarrow } \cdots \end{equation}
a doubly infinite minimal totally acyclic $R$-complex, 
we will construct a doubly infinite minimal totally acyclic $S$- complex
\begin{equation}\label{complex}
\cdots \longrightarrow S^{2b_{i+1}} \stackrel{\epsilon_{i+1}}{\longrightarrow}S^{2b_i}\stackrel{\epsilon_i}{\longrightarrow} \cdots
\end{equation}
Any cokernel in the complex (\ref{complex}) will be a non-free totally reflexive $S$-module.
Let $\tilde{\delta_i}: S^{b_i} \rightarrow S^{b_{i-1}}$ denote a lifting of $\delta_i$ to $S$, for all $i \in {\bf Z}$. We will view these maps as matrices with entries in $S$. Since $\delta_{i}\delta_{i+1}=0$, it follows that there exists a matrix $M_{i+1}$ with entries in $S$ such that 
\begin{equation}\label{matrix} \tilde{\delta}_i\tilde{\delta}_{i+1}= xM_{i+1}. \end{equation}
We define $\epsilon_i$ as follows: if $i$ is even, then 
$$
\epsilon_i = \left[ \begin{array}{cc} \tilde{\delta}_i & xI_{b_{i-1}} \\ M_i & \tilde{\delta}_{i-1} \\ \end{array}\right], 
$$
If $i$ is odd, 
$$
\epsilon_i= \left[ \begin{array}{cc} \tilde{\delta}_i & -xI_{b_{i-1}} \\ -M_i & \tilde{\delta}_{i-1}\\ \end{array} \right]
$$
\end{construction}
Note that if all the entries of $\delta_i$ are in the homogeneous maximal ideal of $R$ for all $i$, then all the entries of $\epsilon_i$ will be in the homogeneous maximal ideal of $S$ (since $x$ has degree one and the entries of $\tilde{\delta_i}\tilde{\delta}_{i+1}$ have degree at least two,  equation (\ref{matrix}) shows that the entries of $M_i$ cannot be units).
We check that (\ref{complex}) is a complex. Let $i$ be even. We have
$$
\epsilon_i \epsilon_{i+1} = \left[ \begin{array}{cc}  \tilde{\delta}_i \tilde{\delta}_{i+1} - xM_{i+1} & -x\tilde{\delta}_i + x \tilde{\delta}_i \\ M_i \tilde{\delta}_{i+1} - \tilde{\delta}_{i-1} M_{i+1} & -xM_i + \tilde{\delta_{i-1}}\tilde{\delta}_i \\ \end{array}\right]
$$
Using equation (\ref{matrix}), we see that all entries are zero except possibly $M_i \tilde{\delta}_{i+1}-\tilde{\delta}_{i-1} M_{i+1}$. However, we have
$$
x(M_i\tilde{\delta}_{i+1}-\tilde{\delta}_{i-1}M_{i+1})= (xM_i)\tilde{\delta}_{i+1} - \tilde{\delta}_{i-1}(xM_{i+1})= \tilde{\delta}_{i-1} \tilde{\delta}_i \tilde{\delta}_{i+1} - \ \tilde{\delta}_{i-1} \tilde{\delta}_i \tilde{\delta}_{i+1} =0.
$$
The calculation is similar if $i$ is odd. 

Now we check that the complex (\ref{complex}) is exact. Let $i$ be even, and let $c=[c_1, c_2]^t \in \mathrm{ker}(\epsilon_i)$, where $c_1 \in S^{b_i}$ and $c_2 \in S^{b_{i-1}}$. We have $\tilde{\delta_i} c_1 +xc_2=0$, and therefore $\delta_i(\overline{c_1})=0$. The exactness of (\ref{complex1}) implies that there are elements $d_1, d_2 \in S$ such that $c_1=\tilde{\delta}_{i+1}d_1 - xd_2$.

Define
$$
\left[\begin{array}{c} c_1' \\ c_2' \\ \end{array}\right] := \left[\begin{array}{c} c_1 \\ c_2 \\ \end{array}\right] - \epsilon_{i+1} \left[ \begin{array}{c} d_1 \\ d_2 \\ \end{array}\right].
$$
It is clear that $c_1 ' =0$ and $[c_1', c_2']^t \in \mathrm{ker}(\epsilon_i)$.
It follows that $xc_2'=0$. Since $x \in S$ is a regular element, we must have $c_2'=0$. In other words, $[c_1, c_2]^t= \epsilon_{i+1}[d_1, d_2]^t\in \mathrm{im}(\epsilon_{i+1})$, which is what we wanted to show.

The calculation is similar if $i$ is odd. 

We also need to check that the dual of the complex (\ref{complex}) is exact.
Let $i$ be even and let $[c_1, c_2]^t \in \mathrm{ker}(\epsilon_{i+1}^t)$, where $c_1 \in S^{b_i}$ and $c_2 \in S^{b_{i-1}}$.
We have 
$$
\epsilon_{i+1}^t=\left[ \begin{array}{cc} \tilde{\delta}_{i+1} ^t &  -M_{i+1}^t\\ -xI_{b_i} & \tilde{\delta}_i^t \\ \end{array}\right].
$$
It follows that $-xc_1 + \tilde{\delta}_i^t c_2=0$, so $\delta_i^t(\overline{c_2})=0.$ 
Due to the exactness of the dual of (\ref{complex1}), we have $c_2 = xd_1 + \delta_{i-1}^t d_2$ for some $d_1, d_2 \in S$.
Define
$$
\left[\begin{array}{c} c_1' \\ c_2'\\ \end{array}\right] := \left[ \begin{array}{c} c_1 \\ c_2 \\ \end{array}\right] - \epsilon_i^t \left[ \begin{array}{c} d_1 \\ d_2 \\ \end{array}\right].
$$
It is clear that $c_2'=0$ and $[c_1', c_2']^t \in \mathrm{ker}(\epsilon_{i+1}^t)$. Therefore, we have $xc_1' =0$. Since $x \in S$ is a regular element, it follows that $c_1'=0$, and thus $[c_1, c_2]^t = \epsilon_i^t [d_1, d_2]^t \in \mathrm{im}(\epsilon_i^t)$, which is what we wanted to show.

The calculation is similar if $i$ is odd.

\section{Stanley-Reisner rings of graphs}

Let $\Gamma = (V, E)$ be a connected graph, where  $V=\{x_1, \ldots, x_n\}$ is the set of vertices, and $E$  is the set of edges. Let $k$ be an infinite field. The Stanley-Reisner ring of $\Gamma$ over $k$ is
$$
R_{\Gamma}=\frac{k[X_1, \ldots, X_n]}{I_{\Gamma}}
$$
where $I_{\Gamma}$ is the ideal generated by all the monomials $X_iX_j$ for which $\{x_i, x_j\}\notin E$, and all monomials $X_iX_jX_k$ with distinct $i, j, k$. 
The general theory of Stanley-Reisner rings (see \cite{BH}, Corollary 5.3.9) shows that, under the assumption that $\Gamma$ is connected, $R_{\Gamma}$ is a two-dimensional Cohen-Macaulay ring. We investigate the existence of non-free  totally reflexive modules for $R_{\Gamma}$ via reducing modulo a linear system of parameters.
We denote $|V|=n$ and $|E|=e$.
\begin{obs}\label{calcul}
 Let $l_1, l_2 \in k[X_1, \ldots, X_n]$ be general linear forms. 
Then $R:=R_{\Gamma}/(l_1, l_2)$ is an Artinian ring with maximal ideal $\mathfrak{m}$. We have $\mathfrak{m}^3=0$, $\mathrm{dim}_k \frak{m}/\frak{m}^2 = n-2$, and $\mathrm{dim}_k \mathfrak{m}^2=e-n+1$.
\end{obs}
\begin{proof}
Note that the degree two component of $R_{\Gamma}$ is generated by $X_1^2, \ldots, X_n^2$, and $X_iX_j$ with $\{x_i, x_j \} \in E$, and the degree three component is generated by $X_1^3, \ldots, X_n^3$, and $X_i^2X_j, X_iX_j^2$  with $\{x_i, x_j\} \in E$.
Therefore, the Hilbert series of $R_{\Gamma}$ has the form
$$
H_{R_{\Gamma}}(t) = 1+nt + (n+e)t^2+ (n+2e)t^3+\cdots 
$$
Since the images of two general linear form $l_1, l_2$ are a regular sequence in $R_{\Gamma}$, we have 
$$
H_R(t)=(1-t)^2 H_{R_{\Gamma}}(t)= 1+(n-2)t+(e-n+1)t^2 +0t^3
$$
which proves the claim.
\end{proof}
\begin{obs}
In order for $\Gamma$ to be connected, we must have $e \ge n-1$. We have $\mathfrak{m}^2=0\Leftrightarrow e=n-1 \Leftrightarrow \Gamma$ is a tree. 
\end{obs}
Note that if $\m^2=0$, then it is known that $R$ does not have non-free totally reflexive modules (see \cite{Yoshino}).
In \cite{Yoshino}, Yoshino gives the following necessary conditions for an Artinian ring with $\m^3=0$ to have non-free totally reflexive modules:
\begin{theorem}\label{Yoshino}\rm(\cite{Yoshino}, Theorem 3.1)
Let $(R, \m)$ be a non-Gorenstein local ring with $\m ^3=0$. Assume that $R$ contains a field $k$ isomorphic to $R/\m$, and assume that there is a non-free totally reflexive $R$-module $M$. Then:

(1) $R$ has a natural structure of homogeneous graded ring with $R=R_0 \oplus R_1 \oplus R_2$ with $R_0=k$, $\mathrm{dim}_k(R_1)=r+1$, and $\mathrm{dim}_k(R_2)=r$, where $r$ is the type of $R$. Moreover, $(0:_R \m)=\m^2$. 

(2) $R$ is a Koszul algebra.

(3) $M$ has a natural structure of graded $R$-module, and, if $M$ is indecomposable, then the minimal free resolution of $M$ has the form
$$
\cdots \rightarrow R(-n-1)^b \rightarrow R(-n)^b \rightarrow \cdots \rightarrow R(-1)^b \rightarrow R^b \rightarrow M \rightarrow 0
$$.
In other words, the resolution of $M$ is linear with constant betti numbers.
\end{theorem}

Based on Yoshino's result, we conclude that the following are necessary conditions for $R_{\Gamma}$ to have non-free totally reflexive modules:
\begin{obs}
If  $R_{\Gamma}$ has non-free totally reflexive modules, then the following must hold:

{\rm (a).}  $ e=2n-4$.

{\rm (b).} $\Gamma$ does not have any cycles of length 3. 

{\rm (c.)} $\Gamma$ does not have  leaves (a leaf is a vertex  which belongs to only one edge).
\end{obs}
\begin{proof} 
If $R_{\Gamma}$ has non-free totally reflexive modules, then so does $R_{\Gamma}/(l_1, l_2)$. We apply the necessary conditions from Theorem (\ref{Yoshino}) to the ring $R=R_{\Gamma}/(l_1, l_2)$.
Part (a) is immediate using the calculation from Observation (\ref{calcul}). Part (b) is a consequence of the requirement that $R$ is a Koszul algebra, which in particular implies that it has to be defined as a quotient of a polynomial ring by quadratic equation. If the graph $\Gamma$ has a cycle consisting of vertices $x_k, x_l, x_j$, then $x_kx_lx_j$ would be one of the defining equations of the Stanley-Reisner ring, and also one of the defining equations of $R_{\Gamma}/(l_1, l_2)$ (viewed as a quotient of a polynomial ring in two fewer variables). To see (c)., assume that there is a vertex $x_k$ of $\Gamma$ that belongs to only one edge, say $\{ x_k, x_j\}$. We can use the equations $l_1, l_2$ to replace the variables $X_j, X_k$ by linear combinations of the remaining variables, and view $R=R_{\Gamma}/(l_1, l_2)$ as a quotient of a polynomial ring in these remaining variables. Then the image of $X_k$ annihilates the images of all the variables, and therefore $\overline{X_k}\in (0:_R \m)$. This contradicts the condition $(0:_R \m)=\m^2$ from Theorem (\ref{Yoshino}).

\end{proof}

We point out that conditions (1) and (2) in Theorem (\ref{Yoshino}) are far from being sufficient for the existence of non-free totally reflexive modules in the case of rings with $\m^3=0$. We will give examples that satisfy these conditions, but do not have totally reflexive modules (see Proposition \ref{no_totally_reflexive}).

We will be able to obtain better results when the graph $\Gamma$ is bipartite, i.e. the vertices can be labeled $x_1, \ldots, x_k, y_1, \ldots, y_l$, and all the edges are of the form $\{x_i, y_j\}$ for some $i, j$. This in particular implies that the graph does not have cycles of length three (in fact, a graph is bipartite if and only if it does not have any cycles of odd length). The following observation is easy to check:

\begin{obs}\label{prime}
Let $\Gamma$ be a bipartite graph, and let 
$$l_1= \sum_{i=1}^k X_i,  l_2=\sum_{j=1}^l Y_j.$$
Then $l_1, l_2$ is a system of parameters. $R=R_{\Gamma}/(l_1, l_2)$ can be regarded as a quotient of $k[X_1, \ldots, X_{k-1}, Y_1, \ldots, Y_{l-1}]$, and it satisfies
\begin{equation}\label{squares}(\overline{X}_, \ldots, \overline{X}_{k-1})^2=(\overline{Y}_1, \ldots, \overline{Y}_{l-1})^2=0 \end{equation}
where $\overline{X_i}, \overline{Y_j}$ denote the images of $X_i, Y_j$ in $R$.

 For every $u \in R_1$, we write $u:=x+y$, where $x$ is a linear combination of $\overline{X}_1, \ldots, \overline{X}_{k-1}$, and $y$ is a linear combination of $\overline{Y}_1, \ldots, \overline{Y}_{l-1}$. We define $u':= x-y$, and observe 
\begin{equation} \label{zero} uu'=0 . \end{equation}
\end{obs}
\begin{proof}
Since all the edges of the graph are of the form $\{ x_i, y_j\}$, it follows that the products of the images in $R_{\Gamma}$ of any two distinct $X_i, X_j$ is zero. Moreover, the equation $X_iX_k=0$ in $R_{\Gamma}$ translates to $\overline{X}_i(\sum_{j=1}^{k-1} \overline{X}_j)=0$ in $R$, and thus we obtain $\overline{X}_i^2=0$ in $R$. A similar argument shows that $\overline{Y}_j^2=0$, and we obtain (\ref{squares}). Now (\ref{squares}) implies that the product of the images of any three variables in $R$ is zero, and therefore $R$ satisfies $\m^3=0$. Since $\mathrm{dim}(R_{\Gamma})=2$ and $R$ is Artinian, it follows that $l_1, l_2$ is a system of parameters for $R_{\Gamma}$.
The claim (\ref{zero}) is obvious.
\end{proof}

We observe that in the case of graded  rings with $\m^3=0$ and $\mathrm{dim}_k(R_2)=\mathrm{dim}_k(R_1)-1$, there is a connection between existence of exact zero divisors and the Weak Lefschetz Property (WLP). In this case, WLP simply means that there exists an element $x \in R_1$ such that the multiplication by $x$ map $:R_1 \rightarrow R_2$ has maximal number of generators, i.e. it is surjective. See (\cite{MN}) for information regarding the  WLP.
\begin{obs}\label{wlp}
\rm{(a.)} Let  $R=k \oplus R_1 \oplus R_2$  be a standard graded ring with $R_3=0$ and $\mathrm{dim}_kR_2 = \mathrm{dim}_kR_1 -1$.  If $R$ admits a pair of exact zero divisors $x, y$, then $R$ has WLP.

\rm{(b.)} Assume that $R=R_{\Gamma}/(l_1, l_2)$, where $\Gamma$ is a bipartite graph, and $l_1, l_2$ are as in Observation (\ref{prime}). If $R$ has WLP, then $R$ admits a pair of exact zero divisors.
\end{obs}
\begin{proof}
(a).  Assume that $(x, y)$ is a pair of exact zero-divisors. By  Theorem (\ref{Yoshino}), we have $x, y \in R_1$.  Then the kernel of the map $\cdot x: R_1 \rightarrow R_2$ is generated by $y$, and is therefore 1-dimensional as a $k$-vector space. It follows that the dimension of the image is $\mathrm{dim}_kR_1 -1=\mathrm{dim}_k(R_2)$, and therefore the map is surjective.

(b) Assume that $z \in R_1$ is such that $\cdot z : R_1 \rightarrow R_2$ is surjective. Equivalently, the kernel of this map is a one-dimensional vector space. Using the notation from Observation (\ref{prime}), we have $zz'=0$. Therefore, every element in $R_1$ that annihilates $z$ must be a scalar multiple of $z'$. We claim that $\mathrm{Ann}_R(z)=(z')$, which will then imply that $z, z'$ is a pair of exact zero divisors. It suffices to prove that $R_2 \subseteq (z')$, or, equivalently, the map $\cdot z': R_1 \rightarrow R_2$ is surjective. We can write $z=x+y$ and $z'=x-y$ as in Observation (\ref{prime}). We observe that the map $\cdot z:R_1 \rightarrow R_2$ is surjective if and only if $R_2$ is spanned by $x\overline{Y}_1, \ldots, x\overline{Y}_{l-1}, y\overline{X}_1, \ldots, y\overline{X}_{k-1}$, and the same conclusion holds for the map $\cdot z': R_1 \rightarrow R_2$. Therefore, the multiplication by $z$ map is surjective if and only if the multiplication by $z'$ map is.
 
\end{proof}

One might hope that the converse of the statement in Observation (\ref{wlp}) (a) above is true without the extra assumptions we made in Part (b). The example below shows that this is not the case. 
\begin{example}
Let 
$$R=\frac{k[X, Y]}{(X^2-Y^2, X^2-XY, X^3)}.$$
Then $R$ satisfies the assumptions from Observation (\ref{wlp}) (a.), and $R$ has WLP since the multiplication by $ax +by: R_1 \rightarrow R_2$ is surjective as long as $a+b \ne 0$. However, $R$ has a linear socle element, namely $x+y$, which implies that $R$ cannot have exact zero divisors; in fact it cannot have totally reflexive modules, by Theorem (\ref{Yoshino})(1).
\end{example}

Now we give a sufficient condition on a graph $\Gamma$ with $e=2n-4$ for $R=R_{\Gamma}/(l_1, l_2)$ to have WLP. When $\Gamma$ is a bipartite graph satisfying this condition,  Observation (\ref{wlp})(b) implies that $R$ will have a pair of exact zero divisors, and Theorem (\ref{main_theorem}) will then allow us to conclude that $R_{\Gamma}$ has non-free totally reflexive modules. 
\begin{prop}\label{exist_ezd}
\rm{a}. Let $\Gamma$ be a graph with vertex set $V=\{x_1, \ldots, x_n\}$. Assume that $e=2n-4$ and the vertices can be ordered  in such a way that for  each $i\ge 3$, there are at least two edges connecting $x_i$ to $\{x_1, \ldots, x_{i-1}\}$.  
Then $R=R_{\Gamma}/(l_1, l_2)$ has WLP for $l_1, l_2$ a  system of parameters consisting of linear forms with generic coefficients. 

\rm{b}. Assume moreover that  $\Gamma$ is bipartite with vertex set $\{x_1, \ldots, x_k, y_1, \ldots, y_l\}$ (where $n=k+l$),  and $l_1=\sum_{i=1}^kx_i, l_2=\sum_{j=1}^ly_j$. Then $R=R_{\Gamma}/(l_1, l_2)$ admits a pair of exact zero-divisors.
\end{prop}
\begin{proof}
\rm{a}. The calculation of the Hilbert function of $R$ from Observation (\ref{calcul}) shows that the map $\cdot l:R_1 \rightarrow R_2$ has maximal number of generators if and only if it is surjective, which is equivalent to having one dimensional kernel.
  Fix $l_1=\sum_{i=1}^n \alpha_i x_i, l_2=\sum_{i=1}^n \beta _i x_i$, and $l=\sum_{i=1}^n a_i x_i$, where the coefficients $\alpha_i, \beta_i, a_i$ are generic in $k^{3n}$.
We consider the linear forms $f_1 =\sum_{i=1}^n u_i x_i, f_2=\sum_{i=1}^n v_i x_i$, and $f =\sum_{i=1}^n w_i x_i$ satisfying 
\begin{equation}\label{kernel}
l_1f_1+l_2f_2+lf=0 \ \mathrm{in}\ R_{\Gamma}
\end{equation}
Equation (\ref{kernel}) translates into a system of $e+n$ equations in $3n$ unknowns. The unknowns are the coefficients $u_i, v_i, w_i$ for $i=1, \ldots, n$, and we get one equation corresponding to each edge $\{x_i, x_j\}$ of $\Gamma$:
\begin{equation}\label{coeff1}
\alpha_ju_i + \alpha_i u_j+\beta_jv_i + \beta_i v_j + a_jw_i + a_i w_j=0,
\end{equation}
obtained by setting the coefficient of $x_ix_j$ in equation (\ref{kernel}) (these account for $2n-4$ equations), and one equation for each $i =1, \ldots, n$:
\begin{equation}\label{coeff2}
\alpha_iu_i+\beta_iv_i+a_iw_i=0,
\end{equation}
which is obtained by setting the coefficient of $x_i^2$ in equation (\ref{kernel}) equal to zero. 
We claim that if the coefficients $\alpha _i, \beta_i, a_i$ are chosen generically, then the vector space of solutions this system of linear equations is four dimensional. Equations (\ref{coeff2}) give $\displaystyle w_i = -\frac{\alpha_i}{a_i} u_i - \frac{\beta_i}{a_i} v_i$. Plugging this into the equations (\ref{coeff1}), we obtain $2n-4$ equations with $2n$ unknowns, of the form
\begin{equation}\label{newsystem}
\alpha_{ji}u_i+\alpha_{ij}u_j + \beta_{ji}v_i + \beta _{ij}v_j=0
\end{equation}
for each edge $\{x_i, x_j\}$ in $\Gamma$, where $$ \alpha_{ij}=\frac{\alpha_ia_j-\alpha_ja_i}{a_j}, \alpha_{ji}=\frac{\alpha_ja_i-\alpha_ia_j}{a_i}, \beta_{ij}=\frac{\beta_ia_j-\beta_ja_i}{a_j}, \beta_{ji}=\frac{\beta_ja_i-\beta_ia_j}{a_i}.$$
	By assumption $\{x_1, x_3\}$ and  $\{x_2, x_3\}$ are edges. The two equations corresponding to these edges involve 6 unknowns, $u_i, v_i$ for $i=1, 2, 3,$. The two equations in (\ref{newsystem}) corresponding to the edges $\{x_1, x_3\}, \{x_2, x_3\}$ allow us to solve for $u_3, v_3$ as linear combinations of $u_1, v_1, u_2, v_2$ (using Cramer's rule, provided the determinant $\alpha_{13}\beta_{23}-\beta_{13} \alpha_{23}$ is nonzero). Now let $i\ge 3$. By induction, we may assume that $u_j, v_j$ can be expressed as linear combinations of $u_1, v_1, u_2, v_2$ for all $j \le i-1$. By assumption,  there are two edges that connect $x_i$ to the set $\{x_1, \ldots, x_{i-1}\}$. Say that these edges are $\{x_{i_1}, x_i\}$ and $\{x_{i_2}, x_i\}$.  The equations in (\ref{newsystem}) corresponding to these edges allow us to solve for $u_i, v_i$ in terms of $u_{i_1}, v_{i_1}, u_{i_2}, v_{i_2}$ (using Cramer's rule, provided that the determinant $\alpha_{i_1i}\beta_{i_2i}-\beta_{i_11}\alpha_{i_2i}$ is nonzero), and therefore in terms of $u_1, v_1, u_2, v_2$ using the inductive hypothesis. It is immediate to see that the conditions 
\begin{equation}\label{det}
\alpha_{i_1i}\beta_{i_2i}-\beta_{i_1i}\alpha_{i_2i} \ne 0
\end{equation} translate into non-vanishing of certain non-trivial polynomials in $\alpha_i, \beta_i, a_i$, and thus there is a non-empty open set in $k^{3n}$ such that for any choice of $\alpha_i, \beta_i, a_i$ in this open set, the vector space of solutions of (\ref{kernel}) is four dimensional.

Now observe that three of the solutions of equation (\ref{kernel}) come from the Koszul relations on $l_1, l_2, l$, so $(f_1^1, f_2^1, f^1)=(-l_2, l_1, 0), (f_1^2, f_2^2, f^2)=(-l,  0, l_1), (f_1^3, f_2^3, f^3)=(0, -l, l_2)$ are linearly independent solutions. Let $(f_1^4, f_2^4, f^4)$ be such that $(f_1^j, f_2^j, f^j)$ where $j=1, \ldots, 4$ is a basis for the vector space of solutions of (\ref{kernel}). Consider the map $\phi$ given by multiplication by the image of $l:R_1 \rightarrow R_2$, where $R=R_{\Gamma}/(l_1, l_2)$. For a linear form $f \in R_{\Gamma}$, the image of $f$ is in the kernel of this map if and only if there exist $f_1, f_2 \in R_{\Gamma}$ such that $(f_1, f_2, f)$ is a solution to (\ref{kernel}). This implies that $f \in (l_1, l_2, f^4)$. Therefore, the kernel of the $\phi$ is one-dimensional, spanned by the image of $f^4$, and thus $\phi$ is surjective.

(b) We need to check that the choice of $l_1=\sum_{i=1}^k x_i, l_2=\sum_{j=1}^l y_j$ allows us to choose $l=\sum_{i=1}^k a_ix_i+\sum_{j=1}^l a_j' y_j$ such that the determinants in (\ref{det}) are non-zero. With notation as above, we have $\alpha_i=1, \beta_i=0$  for $1\le i \le k$, $\alpha_i=0, \beta_i=1$ for $k+1 \le i \le n$. The conditions (\ref{det}) need to be checked whenever $\{i_1, i\}$ and $\{i_2, i\}$ are edges of $\Gamma$. Due to the bipartite nature of the graph, this means that we have either $i_1, i_2 \le k$ and $i\ge k+1$, or $i_1, i_2 \ge k+1$ and $i \le k$. In the first case, we have  
$\displaystyle \alpha_{i_1i}=\frac{a_{i_1}}{a_i}, a_{i_2i}=\frac{a_{i_2}}{a_i}, \beta_{i_1i}=\beta_{i_2i}=0$, and (\ref{det}) becomes $a_{i_1}\ne a_{i_2}$. The second case is similar.

\end{proof}
Now we  give a condition for a bipartite graph $\Gamma$ that implies that the ring $R=R_{\Gamma}/(l_1, l_2)$ does not have exact zero divisors. This will be used in the next section to construct an example of a ring that has no exact zero divisors, but has non-free totally reflexive modules.
\begin{prop}\label{no_ezd}
Let $\Gamma$ be a bipartite graph with vertex set $V=\{x_1, \ldots, x_k, y_1, \ldots, y_l\}$, with $e=2n-4$. Assume that there exist $i \in \{1, \ldots, k\}$ and $j \in \{1, \ldots, l\}$ such that the subgraph induced on $V \, \backslash \{x_i, y_j\}$ is disconnected. Then $R=R_{\Gamma}/(l_1, l_2)$ does not have exact zero divisors.
\end{prop}
\begin{proof}
Without loss of generality, we may assume $i=k, j=l$. Then the maximal ideal of $R$ is generated by the images  $\overline{X}_1, \ldots, \overline{X}_{k-1}, \overline{Y}_1, \ldots, \overline{Y}_{l-1}$ of the variables corresponding to the vertices in $V \, \backslash \{x_k, y_l\}$  . Since the graph induced on $V \, \backslash \{x_k, y_l\}$ is disconnected, we may partition the set of vertices into disjoint sets $A, B$ such that there is no edge connecting any vertex of $A$ to any vertex of $B$. Let $\mathfrak{a}, \mathfrak{b}$ denote the ideals of $R$ generated by the images of the variables corresponding to vertices in $A$ and $B$ respectively. Then we have $\mathfrak{a} \mathfrak{b}=0$, and $\mathfrak{a}+ \mathfrak{b}=\mathfrak{m}$. Assume that $u, v$ is a pair of exact zero divisors consisting of linear elements in $R$. Then we can write $u=u_{\mathfrak{a}}+u_{\mathfrak{b}}$ with $u_{\mathfrak{a}} \in \mathfrak{a}$ and $u_{\mathfrak{b}} \in \mathfrak{b}$. Using the notation from Observation (\ref{prime}), 
we also have $u'_{\mathfrak{a}} \in \mathfrak{a}$, $u'_{\mathfrak{b}} \in \mathfrak{b}$. Then we have
$u u'_{\mathfrak{a}}= u_{\mathfrak{a}} u'_{\mathfrak{a}} + u_{\mathfrak{b}} u'_{\mathfrak{a}}=0$, where the first term is zero from (\ref{zero}), and the second term is zero because $\mathfrak{a} \mathfrak{b}=0$. Similarly, we have $u u'_{\mathfrak{b}}=0$. This shows that $u$ cannot be part of a pair of exact zero-divisors. 

\end{proof}

\begin{obs}\label{no_totally_reflexive}
If we replace the property that $\m=\mathfrak{a}+\mathfrak{b}$ and $\mathfrak{a} \mathfrak{b}=0$ from the proof of  Proposition (\ref{no_ezd}) with the stronger condition that there are non-zero ideals
 $\mathfrak{a}, \mathfrak{b}\subseteq \mathfrak{m}$ such that $\mathfrak{m}=\mathfrak{a}\oplus \mathfrak{b}$, then the ring $R$ is a fiber product, and it follows by \cite{SS} that it does not have any non-free totally reflexive modules unless it is Gorenstein. This holds without the assumption that $\m ^3=0$, but we provide a more direct proof for this case below.
\end{obs}
\begin{proof}
Assume that $R$ is not Gorenstein and has non-free totally reflexive modules.
Condition (1) from Theorem (\ref{Yoshino}) tells us that we may assume $\nu(\m) \ge 3$.
We know from Theorem (\ref{Yoshino}) that the resolution of a totally reflexive module must have constant Betti numbers $B$ and matrices $D_i$ consisting of linear forms in $R$. 
We can write $D_i = D_{i, \mathfrak{a}} + D_{i, \mathfrak{b}}$ where $D_{i, \mathfrak{a}}$ has entries in $\mathfrak{a}$, and $D_{i, \mathfrak{b}}$ has entries in $\mathfrak{b}$. 
Every vector $\mathfrak{u}$ with entries in $\mathfrak{m}$ can be written as $u_{\mathfrak{a}}+ u_{\mathfrak{b}}$, where $u_{\mathfrak{a}}$ has entries in $\mathfrak{a}$ and $u_{\mathfrak{b}}$ has entries in $\mathfrak{b}$. 
Note that
$$ D_i u = D_{i, \mathfrak{a}}u_{\mathfrak{a}} + D_{i, \mathfrak{b}}u_{\mathfrak{b}},
$$
and $u \in \mathrm{\ker}(D_i) \Leftrightarrow u_{\mathfrak{a}} \in \mathrm{ker}(D_{i, \mathfrak{a}})$ and $u_{\mathfrak{b}} \in \mathrm{ker}(D_{i, \mathfrak{b}})\Leftrightarrow u_{\mathfrak{a}}, u_{\mathfrak{b}} \in \mathrm{ker}(D_i)$. Since the columns of $D_{i+1}$ span $\mathrm{ker}(D_i)$, it follows that we can write $D_{i+1}$ as a matrix in which every column has either all entries in $\mathfrak{a}$ or all entries in $\mathfrak{b}$. Say that there are $n_i$ columns of the first type, and $n'_i$ columns of the second type, where $n_i + n'_i=B$. The columns of $D_{i+1}$ that have all entries in $\mathfrak{a}$ are annihilated by evey element of $\mathfrak{b}$, and the columns of $D_{i+1}$ that have all entries in $\mathfrak{a}$ are annihilated by every element of $\mathfrak{b}$. Say that $\nu (\mathfrak{a})= a$ and $\nu (\mathfrak{b})=b$. Then we have $bn_i +an'_i$ linearly independent relations on the columns of $D_{i+1}$ described in the previous sentence. It follows that $B\ge bn_i + an'_i$. Since $n_i + n'_i=B$, and $a, b \ge 1$,  this is only possible if $a=b=1$. This would contradict the assumption that $\nu(\mathfrak{m})\ge 3$.

\end{proof}
Examples of rings $R$ satisfying the hypothesis in Observation~\ref{no_totally_reflexive} are obtained from bipartite graphs $\Gamma$ that satisfy the assumption in Proposition ~\ref{no_ezd} for some $i, j$ , and also have the property that $x_i$ is connected to all of $y_1, \ldots, y_l$, and $y_j$ is connected to all of $x_1, \ldots, x_k$.

\begin{obs}
Given a bipartite graph $\Gamma$ that satisfies $e=2n-4$, one is led to wonder whether one of the hypothesis in Proposition (\ref{exist_ezd}) or the hypothesis in Proposition (\ref{no_ezd}) must hold. We have not been able to establish this or find a counterexample.
\end{obs}

\section{An example with no exact zero divisors}
In this section we study an example of a ring $(R, \mathfrak{m})$ with $\m^3=0$ such that $\mathfrak{m}=\mathfrak{a}+\mathfrak{b}$ for two ideals $\mathfrak{a}, \mathfrak{b}\subseteq \mathfrak{m}$ which satisfy $\mathfrak{a}\mathfrak{b}=(0)$, but $\mathfrak{a} \cap \mathfrak{b}\ne (0)$. Proposition (\ref{no_ezd}) can be applied to show that this ring does not have exact zero divisors. We will give a construction that produces infinitely many non-isomorphic indecomposable totally reflexive modules over this ring. It is theoretically known for a non-Gorenstein ring that if it has one non-free totally reflexive module, then it must have infinitely many non-isomorphic indecomposable totally reflexive modules;  see \cite{CPST}.  However, most concrete constructions that give rise to infinitely many such modules rely on the existence of a pair of exact zero divisors;  see \cite{CJRSW}, \cite{Tr}. The example we study here shows how such a construction can be achieved in the absence of exact zero divisors.

Throughout this section, $R$ will denote the ring described below.
\begin{construction}\label{part1}
Let $\Gamma$ be the bipartite graph with vertices
\newline $\{x_1, \cdots , x_5, y_1, \cdots, y_5\}$ and edges
$\{x_1, y_1\}, \{x_1, y_2\}, \{x_2, y_1\}, \{x_2, y_2\}, \{x_3, y_3\}$, $\{x_3, y_4\}$, $ \{x_4, y_3\}$, $\{x_4, y_4\}$, and
$\{x_i, y_5\}, \{x_5, y_j\}$ for all $i, j \in \{1, 2, 3, 4\}$.
\end{construction}

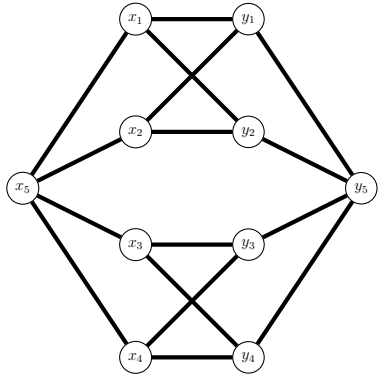
\begin{center}
\begin{tikzpicture}[scale=1.0]
\node at (0, 0) (i) [circle, draw, scale=0.5, fill=white] {$x_4$}; 
\node at (1.5, 0) (j) [circle, draw, scale=0.5,  fill=white] {$y_4$};
\node at (1.5, 1.5) (ell) [circle, draw, scale=0.5, fill=white] {$y_3$};
\node at (0, 1.5) (m) [circle, draw, scale=0.5, fill=white] {$x_3$};
\node at (0, 3) (a) [circle, draw, scale=0.5,   fill=white] {$x_2$}; 
\node at (1.5, 3) (b) [circle, draw, scale=0.5,   fill=white] {$y_2$};
\node at (1.5, 4.5) (c) [circle, draw, scale=0.5,   fill=white] {$y_1$};
\node at (0, 4.5) (d) [circle, draw, scale=0.5,  fill=white] {$x_1$};
\node at (-1.5, 2.25) (x) [circle, draw, scale=0.5,  fill=white] {$x_5$};
\node at (3, 2.25) (y) [circle, draw, scale=0.5,  fill=white] {$y_5$};
\draw [-, ultra thick] (a) to (b);
\draw [-, ultra thick] (a) to (c);
\draw [-, ultra thick] (d) to (b);
\draw [-, ultra thick] (d) to (c);
\draw [-, ultra thick] (i) to (j);
\draw [-, ultra thick] (i) to (ell);
\draw [-, ultra thick] (m) to (j);
\draw [-, ultra thick] (m) to (ell);
\draw [-, ultra thick] (x) to (a);
\draw [-, ultra thick] (x) to (d);
\draw [-, ultra thick] (x) to (i);
\draw [-, ultra thick] (x) to (m);
\draw [-, ultra thick] (y) to (b);
\draw [-, ultra thick] (y) to (c);
\draw [-, ultra thick] (y) to (ell);
\draw [-, ultra thick] (y) to (j);
\end{tikzpicture}
\end{center}

Note that removing the vertices $x_5, y_5$ yields a disconnected graph, with connected components $\{x_1, x_2, y_1, y_2\}$ and $\{x_3, y_3, x_4, y_4\}$ (which are complete  bipartite subgraphs). 

The ring $R=R_{\Gamma}/(l_1, l_2)$, where $l_1 = \sum_{i=1}^5 x_i, l_2=\sum_{j=1}^5 y_j$ is 
$$
R=\frac{k[x_1, \cdots, x_4, y_1, \cdots, y_4]}{(x_1, \cdots , x_4)^2 + (y_1, \cdots, y_4)^2 +I},
$$
where $I=(x_1, x_2)(y_3, y_4)+ (x_3, x_4)(y_1, y_2) + ((\sum_{i=1}^4 x_i)(\sum_{j=1}^4y_j)).$
Proposition (\ref{no_ezd}) shows that $R$ does not have exact zero divisors.

Letting $\mathfrak{a}:=(x_1, x_2, y_1, y_2)$ and $\mathfrak{b}:=(x_3, x_4, y_3, y_4)$, we have 
\begin{equation}\label{properties} 
\mathfrak{m}=\mathfrak{a}+\mathfrak{b}, \mathfrak{a}\mathfrak{b}=(0), \mathrm{and} \  \mathfrak{a} \cap \mathfrak{b} = (\delta), 
\end{equation}
where $\delta = (\sum_{i=1}^4 x_i)(\sum_{j=1}^4 y_j)$. The number of vertices of $\Gamma$ is 10 and the number of edges is 16, so the requirement $e=2n-4$ is satisfied. This means that $\mathrm{dim}_k(R_2)=\mathrm{dim}_k(R_1)-1$.

We let $\mathfrak{a}_i$ denote the vector space spanned by monomials of degree $i$ in $x_1, x_2, y_1, y_2$, and $\mathfrak{b}_i$ the vector space spanned by the monomials of degree $i$ in $x_3, x_4, y_3, y_4$ for $i=1, 2$. 

\begin{obs}\label{conditions}
Let $A_0, B_0$ denote $2 \times 2$  matrices of linear forms such that the entries of $A_0$ are in $\mathfrak{a}$ and the entries of $B_0$ are in $\mathfrak{b}$. Assume that the maps $\tilde{A_0}:(\mathfrak{a}_1)^2 \rightarrow (\mathfrak{a}_2)^2$  induced by multiplication by $A_0$ and $\tilde{B_0}: (\mathfrak{b}_1)^2 \rightarrow (\mathfrak{b}_2)^2$ induced by multiplication by $B_0$ are injective.

Consider the map $\tilde{A_0}+\tilde{B_0}: (R_1)^2 \rightarrow (R_2)^2$.
Then $\mathrm{ker}(\tilde{A_0}+\tilde{B_0})$ is generated by two vectors ${\bf c}_1+{\bf d_1}$ and ${\bf c}_2 + {\bf d}_2$ with linear entries,where ${\bf c}_1, {\bf c}_2$ have entries in $\mathfrak{a}$, and ${\bf d}_1, {\bf d}_2$ have entries in $\mathfrak{b}$.

Let $A_1$, $B_1$ denote the matrices with columns ${\bf c}_1, {\bf c}_2$ and ${\bf d}_1, {\bf d}_2$ respectively. If the maps $\tilde{A_1}: (\mathfrak{a}_1)^2 \rightarrow (\mathfrak{a}_2)^2, \tilde{B_1}:(\mathfrak{b}_1)^2 \rightarrow (\mathfrak{b}_2)^2$ are also injective, then we have an exact complex
\begin{equation}\label{exact_complex}
R^ 2 \stackrel{A_1+B_1}{\longrightarrow} R^2 \stackrel{A_0+B_0}{\longrightarrow}R^2.
\end{equation}

\end{obs}
{\bf Note:} We view $\tilde{A_0}, \tilde{B_0}$, etc. as maps of vector spaces, and $A_0, B_0, $ etc. as maps of free $R$-modules.
\begin{proof} 
Note that $\mathfrak{a}_i, \mathfrak{b}_i$ have vector space dimension 4 for $i=1, 2$. Therefore the injectivity assumption implies that $\tilde{A_0}, \tilde{B_0}$ are bijective.
An arbitrary vector in $R^2$ with entries consisting of linear forms can be written as ${\bf c} + {\bf d}$, with  ${\bf c}\in (\mathfrak{a}_1)^2$ and ${\bf d }\in (\mathfrak{b}_1)^2$. Since $A_0 {\bf d}=B_0{\bf c}=0$, we have 
$$
{\bf c} + {\bf d} \in \mathrm{ker}(A_0+B_0) \Leftrightarrow A_0{\bf c} =-B_0{\bf d},
$$
and if that is the case, then the entries of $A_0{\bf c}$ and $B_0{\bf d}$ must be in $(\delta)$, and we have
$$
A_0{\bf c}=-B_0{\bf d}=\left(\begin{array}{c} \alpha \delta \\ \beta \delta \\ \end{array}\right)
$$
with $\alpha, \beta \in k$.
The injectivity assumptions imply that there are unique ${\bf c}_1, {\bf c}_2, {\bf d}_1, {\bf d}_2$ such that 
\begin{equation}\label{Cramer}
A_0{\bf c}_1 =- B_0{\bf d}_1 = \left( \begin{array}{c} \delta \\ 0 \\ \end{array} \right),  \ \ \ \ \ A_0{\bf c}_2 = -B_0{\bf d}_2 = \left( \begin{array}{c} 0 \\ \delta \\ \end{array}\right)
\end{equation}
It is now easy to check that $\mathrm{ker}(\tilde{A_0}+\tilde{B_0})$ is spanned by ${\bf c}_1 +{\bf d}_1, {\bf c}_2 + {\bf d}_2$.

It is clear from construction that (\ref{exact_complex}) is a complex. Recall that $\mathrm{dim}_k(R_2)=\mathrm{dim}_k(R_1)-1$. As above, the injectivity assumptions for $\tilde{A_1}$ and $\tilde{B_1}$ imply that $\tilde{A_1}+\tilde{B_1}:(R_1)^2 \rightarrow (R_2)^2$ has a two dimensional kernel. Since $\mathrm{dim}_k((R_1)^{\oplus 2})=\mathrm{dim}_k((R_2)^{\oplus 2})+2$, it follows that $\tilde{A_1}+\tilde{B_1}$ is surjective. On the other hand, $\mathrm{ker}(A_0+B_0)$ consists of $\mathrm{ker}(\tilde{A_0}+\tilde{B_0})$ in degree one, and all of $R_2^{\oplus 2}$ in degree two. Therefore the surjectivity of $\tilde{A_1}+\tilde{B_1}$, together with the fact that the image of $A_1+B_1$ contains the kernel of $\tilde{A_0}+\tilde{B_0}$  by construction show the exactness of (\ref{exact_complex}).
\end{proof}

\begin{obs}\label{part2}
Assume that there is a doubly infinite sequence of $2\times 2$ matrices $A_n, B_n$ for $n\in {\bf Z}$ with the entries of $A_n$ in $\mathfrak{a}_1$ and the entries of $B_n$ in $\mathfrak{b}_1$, such that $\tilde{A}_n, \tilde{A}_n^t:(\mathfrak{a}_1)^2 \rightarrow (\mathfrak{a}_2)^2$ and $\tilde{B}_n, \tilde{B}_n^t:(\mathfrak{b}_1)^2 \rightarrow (\mathfrak{b}_2)^2$ are injective maps, and $(A_n +B_n)(A_{n+1}+B_{n+1})=0$ for all $n \in {\bf Z}$.

Then we have a doubly infinite acyclic complex
$$
\mathcal {F}_{\cdot} \ \ \ \ \ \ \ \ \ \ \cdots R^2 \stackrel{A_{n+1} + B_{n+1}}{\longrightarrow} R^2 \stackrel{A_n+B_n}{\longrightarrow} R^2 \stackrel{A_{n-1}+B_{n-1}}{\longrightarrow} \cdots 
$$
whose dual is also acyclic. Any cokernel module in $\mathcal{F}_{\cdot}$ will be a non-free totally reflexive $R$-module.
\end{obs}
\begin{proof}
The acyclicity of the complex ${\mathcal F}_{\cdot}$ was proved in Observation (\ref{conditions}). In order to see that the dual is also acyclic, note that Observation (\ref{conditions}) applies to $A_{n+1}^t, B_{n+1}^t$ used in the roles of $A, B$, and therefore the kernel of $\tilde{A}_{n+1}^t+\tilde{B}_{n+1}^t$ is spanned by two vectors with linear entries. Since we know $(A_{n+1}^t+B_{n+1}^t)(A_n^t + B_n^t)=0$, it follows that  $A_n^t, B_n^t$  can be used in the roles of $A_1, B_1$.
\end{proof}

\begin{construction}\label{part3}
Now we provide an explicit construction that satisfies all the required conditions in Observation~(\ref{part2}). Let
$$
A_n=\left(\begin{array}{cc} x_1+x_2+y_1+y_2 & x_1-x_2+y_1-y_2 \\ x_1-x_2+y_1-y_2 & x_1+x_2-y_1-y_2\\ \end{array}\right) ,
$$
$$B_n=\left( \begin{array}{cc} x_3+x_4 +y_3+y_4 & x_3-x_4+y_3-y_4 \\ x_3-x_4+y_3-y_4 & x_3+x_4 -y_3 -y_4 \\ \end{array}\right)
$$
when $n$ is even, and
$$
A_n=\left(\begin{array}{cc} x_1+x_2+y_1+y_2 & x_1-x_2-y_1 +y_2 \\ x_1 -x_2 -y_1 +y_2 & x_1+x_2-y_1-y_2\\ \end{array}\right),
$$
$$
  B_n=\left(\begin{array}{cc} x_3+x_4+y_3+y_4 & x_3-x_4-y_3+y_4 \\ x_3-x_4-y_3+y_4 & x_3+x_4-y_3-y_4 \\ \end{array}\right)
$$
when $n$ is odd. 

All the requirements can be checked by direct calculation.

\end{construction}

\section{Constructing uncountably many totally reflexive modules}

In this section we show that, for graded Cohen-Macaulay  rings $(R, \m)$ that contain the complex numbers and  specialize to a ring with $\m^3=0$ when  modding out a linear regular sequence, if there are non-free totally reflexive modules, then there are uncountably many non-isomorphic indecomposable ones, and they can be constructed as cokernels of matrices with generic linear entries  Note that the existence of infinitely many such modules was known from \cite{CPST}; the uncountablility is an improvement of that statement. This improvement is relevant in view of the theory Dao and Takahashi of radius of a category, applied to the theory of totally reflexive modules over the rings under consideration (see Corollary \ref{en}).

\begin{theorem}\label{uncount}
Let $(R, \m)$ be a standard graded non-Gorenstein ring with $\m^3 =0$ and containing ${\mathbb C}$, the field of complex numbers. Assume that $R$ admits non-free totally reflexive modules. Then there are uncountably many non-isomorphic indecomposable totally reflexive modules.

More precisely, let $R=k[x_1, \dots, x_n]/I$ and let $b$ be the smallest number of generators of a non-free totally reflexive $R$-module. Think of the set of all $b\times b$ matrices with linear entries in $R$ as being parametrized by ${\mathbb C}^{nb^2}$ (each matrix corresponds to the vector which records the coefficients of the linear entries). Then are countably many Zariski open sets $\mathcal{U}_k$  in ${\mathbb  C}^{nb^2}$ such that if $A \in \cap_{k} \mathcal{U}_k$, then $\mathrm{coker}(A)$ is a totally reflexive $R$-module.
\end{theorem}

\begin{proof}
We explain how the first claim in the statement follows from the second. The assumption that $R$ admits a non-free totally reflexive module will imply that the Zariski open sets $\mathcal{U}_k$ are non-empty.
A countable union of proper Zariski closed sets in ${\mathbb C}^{nb^2}$ is a set of measure zero, and therefore its complement is uncountable. The modules we will construct in the proof of the second claim will be syzygies in a totally acyclic complex with constant betti numbers $b$. Thus they are indecomposable (since they have minimal number of generators). We claim that there are uncountably many choices of $A$ that give rise to mutually non-isomorphic cokernels. Let $A$, $A'$ be $b\times b$ matrices with linear entries. Then $\mathrm{coker}(A)\cong \mathrm{coker}(A')$ if and only if there exist invertible $b\times b$ matrices $U,  V$ such that $UA=A'V$. Let $(u)_{ij}, (v)_{ij}$ denote the degree zero components of $U, V$.  Let the $(i, j)$ entry of $A$ be $\sum_{k=1}^n a_{ij}^kx_k$ and the $(i, j)$ entry of $A'$ be $\sum_{k=1}^n a_{ij}^{'k}x_k$. Setting the linear components of the entries of $UA$ equal to those of the entries of $A'V$ and indentifying the coefficients of each $x_k$ gives rise to equations
$$
\sum_{l=1}^b u_{il}a_{lj}^k=\sum_{l=1}^b a_{il}^{'k}v_{lj}
$$
for every $i, j \in \{1, \ldots, b\}$ and every $k \in \{1, \ldots, n\}$. View $u_{ij}, v_{ij}$ as unknown; there are $nb^2$ equations and $2b^2$ unknowns. Since $U, V$ are invertible, this system must have nontrivial solutions. We have $n \ge 3$, since otherwise the Hilbert function of $R$ given by Theorem (\ref{Yoshino}) would force $R$ to be Gorenstein. Therefore, the minors of size $(2b^2+1)\times (2b^2+1)$ of the resulting matrix of coefficients must be zero. These minors are polynomials  in $a_{ij}^k, a_{ij}^{'k}$. Therefore,  if we fix a matrix $A$, then the set of all the matrices $A'$ that have $\mathrm{coker}(A') \cong \mathrm{coker}(A)$ belong to a Zariski closed set in ${\mathbb C}^{nb^2}$. If there were only countably many isomorphism classes of modules obtained as cokernels of matrices in $\cap _k \mathcal{U}_k$, it would follows that ${\mathbb C}^{nb^2}$ can be obtain as a union of countably many proper Zariski closed sets. This is a contradiction.

Now we prove the second claim.
A $b \times b$ matrix with entries in $R$ can be viewed as a $R$-module homomorphism $A: R^b \rightarrow R^b$. Since the entries of $A$ are linear, it also gives rise to a linear map of vector spaces which we denote $\tilde{A}:(R_1)^b \rightarrow (R_2)^b$. We have  $n=\mathrm{dim}_{\bf C} (R_1)$, and,  by Theorem (\ref{Yoshino}), $\mathrm{dim}_{\mathbb C}(R_2)=n-1$. The Zariski open set $\mathcal{U}_0$ is defined as the set of matrices $A$ such that $\tilde{A}: (R_1)^b \rightarrow (R_2)^b$ is surjective, and the columns of $A$ are linearly independent in $(R_1)^b$. We check that this is indeed a Zariski open set. The linear independence of the columns is clearly an open condition. 
We know from Theorem (\ref{Yoshino}) that we can write
$\displaystyle R=\frac{k[x_1, \ldots, x_n]}{(p_1, \ldots, p_r)}$, where $p_1, \ldots, p_r$ are polynomials of degree two.  Let $P_2$ denote the degree two component of the polynomial ring $P=k[x_1, \ldots, x_n]$ and let $E_1, \ldots, E_b$ denote the standard basis vectors in ${\mathbb C}^b$.

$\tilde{A}$ is surjective if and only if the vectors 
\begin{equation}\label{basis} \{x_i\mathrm{col}_j(\tilde{A}), \ p_lE_j \, | \,  i=1, \ldots, n, j=1, \ldots, b, l=1, \ldots, r\} \subseteq  P_2^b\end{equation}
span $P_2^b$. We identify each vector in $P_2^b$ with a vector in ${\mathbb C}^{Nb}$ (by choosing an ordering of the monomials in $P_2$ and  recording the coefficients of each component), where $\displaystyle N=\left(\begin{array}{c} n+1 \\ 2 \\ \end{array}\right)$. Note that $\mathrm{dim}_{\mathbb C}(R_2)=n-1=N-r$. Form a matrix with $bn+br=b(N+1)$ columns and $bN$ rows with entries in ${\mathbb C}$ by recording each vector in (\ref{basis}) as a vector in ${\mathbb C}^{Nb}$ via this identification.  Surjectivity of $\tilde{A}$ translates  into the condition that this matrix has maximal number of generators.  This is obviously an open condition in the coefficients of the entries of $A$.

The Hilbert function of $R$ shows that the surjectivity of $\tilde{A}$ is equivalent to $\mathrm{dim}_{\mathbb C}(\mathrm{ker}(\tilde{A}))=b$. Form a matrix $A_1$ by using a spanning set of $\mathrm{ker}(\tilde{A})$ in $R_1^b$ as columns. $A_1$ is a $b \times b$ matrix with entries consisting of linear forms in $R$.
The Zariski open  set $\mathcal{U}_1$ is defined as the set of matrices $A$ such that $A_1 \in \mathcal{U}_0$. In order to see that this is a Zariski open set, it is enough to check that the entries of $A_1$ are obtained as polynomials in the entries of $A$. This is obvious since the entries of $A_1$ are obtained by solving linear equations with coefficients obtained from the entries of $A$. The Zariski open sets $\mathcal{U}_k$ for $k \ge 2$ are defined recursively as the set of matrices $A$ such that $A_k \in \mathcal{U}_0$, where $A_k$ is defined recursively as the matrix whose columns are a spanning set for $\mathrm{ker}(\tilde{A}_{k-1})$. Also consider the Zariski open sets $\mathcal{U}_k'$ obtained from the transposes of these matrices: $A \in \mathcal{U}_k'\Leftrightarrow A_k^t \in \mathcal{U}_0$.

Now we construct matrices $A_k$ for $k \le -1$. Assume that  $A \in \mathcal{U}'_0$, so the transpose $\tilde{A^t}: (R_1)^b \rightarrow (R_2)^b$ is surjective, and let $A'_{-1}$ denote the $b \times b$ matrix with columns equal to a spanning set for $\mathrm{ker}(\tilde{A^t})$. The coefficients in the linear entries of $A'_{-1}$ can be obtained as polynomials in the entries of $A$. Let $A_{-1}:= (A'_{-1})^t$. Let $\mathcal{U}_{-1}$ denote the Zariski open set of matrices $A$ that yield $A_{-1} \in \mathcal{U}_0$ and let $\mathcal{U}'_{-1}$ denote the Zariski open set of matrices $A$ that yield $A'_{-1} \in \mathcal{U}_0$. The construction is continued recursively: in order to construct $A_{-k-1}$, assume that there are Zariski open sets $\mathcal{U}_{-k}$ and $\mathcal{U}_{-k}'$ such that if $A \in \mathcal{U}_{-k} \cup \mathcal{U}_{-k}'$, then we have $\tilde{A}_{-k}, \tilde{A}^t_{-k}:(R_1)^b \rightarrow (R_2)^b$ surjective. We let $A_{-k-1}'$ denote the matrix with columns obtained as a spanning set of $\mathrm{ker}(A_{-k})^t$, and $A_{-k-1}:=(A_{-k-1}')^t$.

We claim that if $A \in \cap _{k \in {\bf Z}} (\mathcal{U}_k\cap \mathcal{U}_k')$, then there is a doubly infinite complex  consisting of the free modules $R^b$ and differentials given by the matrices $A_k$ , and this complex is totally acyclic.  Let $k \ge 0$. 
We have $$\mathrm{ker}(A_k)=\mathrm{ker}(\tilde{A_k}) \cup (R_2)^b, $$
 and $$\mathrm{im}(A_{k+1})=R \ \mathrm{span\ of\ im}(\tilde{A}_{k+1})= R \ \mathrm{span\  of\ ker}(\tilde{A}_k).
$$
It follows that $A_kA_{k+1}=0$, and exactness follows since the surjectivity of $\tilde{A}_{k+1}$ implies that $(R_2)^b \subseteq \mathrm{im}(A_{k+1})$.

 Now let $j:=-k-1 \le -1$. We want to see that $A_jA_{j+1}=0$. This is equivalent to $A_{-k}^t A_{-k-1}^t=0$. By construction, $A_{-k-1}^t= A_{-k-1}'$, and we have $\mathrm{im}(A_{-k-1}')\subseteq \mathrm{ker}(A_{-k}^t)$, which gives us the desired conclusion. To prove exactness,  we note as before that $\mathrm{ker}(A_j)=\mathrm{ker}(\tilde{A_j}) \cup (R_2)^b$. The choice of $A$ guarantees that $\tilde{A}_{j+1}:(R_1)^n \rightarrow (R_2)^n$ is surjective, so $(R_2)^b \subseteq \mathrm{im}(A_{j+1})$.  Moreover, the subspace of $(R_1)^b$ spanned by the columns of $A_{j+1}$ contains $\mathrm{ker}(\tilde{A}_j)$, and they are both $b$-dimensional.

\end{proof}

\begin{corollar}\label{posdim}
Let $(R, \m)$ be a $d$ dimensional graded Cohen-Macaulay ring over the field of complex numbers ${\mathbb C}$. Assume that $\m^3 \subseteq (x_1, \ldots, x_d)$, where $x_1, \ldots, x_d$ is a linear system of parameters. If $R$ admits non-free totally reflexive modules, then there are uncountably  many mutually non-isomorphic indecomposable totally reflexive modules.
\end{corollar}
\begin{proof}
The case $d=0$ is the content of Theorem (\ref{uncount}). We will prove the case $d=1$. The general case will then follow by induction on $d$. Let $x$ denote a linear parameter and let $R'=R/(x)$. If $M$ is a non-free totally reflexive  $R$-module, then $M'=M/(x)M$ is a non-free totally reflexive $R/(x)$-module. Choose $M$ to have the smallest number of generators among non-free indecomposable totally reflexive $R$-modules. Then $M'$ will also have number of generators $b$. The proof of theorem (\ref{uncount}) shows that we can construct uncountably many mutually non-isomorphic totally reflexive modules with number of generators $b$.  We use the construction (\ref{construction}) from Section 2 to build totally reflexive $R$-modules from each such $R'$ module. We claim that there are uncountably many choices of $M'$ that give rise to mutually non-isomorphic $R$-modules via construction (\ref{construction}). Let $A$, $B$ denote two presentation matrices of non-isomorphic totally reflexive $R'$-modules with number of generators $b$. Assume that $A$ occurs as part of a totally acylic complex as the map $\delta_1: F_1 \rightarrow F_0$, and $B$ occurs as part of a totally acyclic complex as the map $\delta'_0: F'_0 \rightarrow F'_{-1}$. Let $\tilde{A}, \tilde{B}$ denote liftings of $A, B$ to $R$.

Construction (\ref{construction}) gives matrices
$$
E=\left[ \begin{array}{cc}\tilde{A} & -xI_b \\ -C & \tilde{\delta_0} \\ \end{array}\right] \ \ \ \ \ \ \mathrm{and} \ \ F=\left[\begin{array}{cc} \tilde{\delta_1'} & -xI_b \\ -D &\tilde{B} \\ \end{array}\right],
$$
each of which is part of a totally acyclic complex of $R$-modules. As in the proof of theorem (\ref{uncount}), if $\mathrm{coker}(E)\cong \mathrm{coker}(F)$, then there exist invertible matrices $U, V$ such that $EU=VF$. Representing $U, V$ in block form, we can write
$$
U=\left[ \begin{array}{cc} U_1 & U_2 \\ U_2 & U_4 \\ \end{array}\right], \ \ \ \ \ V = \left[\begin{array}{cc} V_1 & V_2 \\ V_3 & V_4 \\ \end{array}\right]
$$
and the requirement that $EU=VF$ implies that $xV_1+V_2\tilde{B}=\tilde{A}U_2+xU_4$. Modulo $x$, we have $\overline{V_2}B=A\overline{U_2}$. The first part of the proof of Theorem (\ref{uncount}) shows that if we choose $A, B$ sufficiently general, then we must have $\overline{\bf v}_2 =\overline{\bf u}_2=0$, where $\overline{\bf u}_2, \overline{\bf v}_2$ denote degree zero components of $U_2, V_2$. We also have $-CU_2+\tilde{\delta}_0U_4=-xV_3+V_4\tilde{B}$. Identifying the linear parts after modding out by $x$, we have $\delta_0  \overline{U}_4=\overline{V}_4B$. Since the coefficients of the linear entries entries of $\delta _0$ can be obtained  as polynomials in terms of the entries of $A$, we may also assume that $\delta _0$ and $B$ are sufficiently general so that this implies that $\overline{\bf u}_4=\overline{\bf v}_4=0$. The fact that $\overline{\bf u_2}=\overline{\bf u_4}=0$ contradicts the assumption that the matrix $U$ is invertible.

Note that the modules given by the construction (\ref{construction}) might not be indecomposable. Nevertheless, having uncountably many mutually non-isomorphic totally reflexive modules of number of generators $b$ implies that there must be uncountably many non-isomorphic indecomposable totally reflexive modules. Otherwise, countably many indecomposables can only give rise to countably many totally reflexive modules.
\end{proof}

\begin{corollar}\label{en}
Let $R$ be as in Corollary (\ref{posdim}). Then the category of totally reflexive $R$-modules is either trivial (consists of just the free modules), or else it has positive radius (see \cite{DT} for the definition of radius of a category).
\end{corollar}
\begin{proof}
The definition of radius of a category in (\cite{DT}, Definition 2.1) counts the minimal number of extensions needed in order to obtain all objects in the cateogory from a single object, via taking direct sums, direct summands, syzygies, and extensions.
The existence of uncountably many totally reflexive $R$-modules implies that they cannot be all obtained from a single object via direct sums, direct summands, and syzygies. Thus, at least one extension is necessary.
\end{proof}

\end{document}